\documentclass[journal]{IEEEtran}

\usepackage{lineno,hyperref}

\usepackage{graphicx} 
\usepackage{amsmath} 
\usepackage{txfonts}
\usepackage{color}
\usepackage{algorithm}
\usepackage{algpseudocode}
\usepackage{comment}
\usepackage[normalem]{ulem}

\newtheorem{theorem}{Theorem}[section]
\newtheorem{lemma}[theorem]{Lemma}
\newtheorem{proposition}[theorem]{Proposition}

\newenvironment{proof}[1][Proof]{\begin{trivlist}
\item[\hskip \labelsep {\bfseries #1}]}{\end{trivlist}}

\begin{document}

\title{Real-time Mode Scheduling Using Single-Integration Hybrid Optimization for Linear Time-Varying Systems}

\title{Real-time Dynamic-Mode Scheduling Using Single-Integration Hybrid Optimization for Linear Time-Varying Systems}

\author{Anastasia~Mavrommati,~\IEEEmembership{Student Member,~IEEE,}
        Jarvis~A.~Schultz,~\IEEEmembership{Member,~IEEE,}
        and~Todd~D.~Murphey,~\IEEEmembership{Member,~IEEE}
\thanks{Anastasia~Mavrommati, Jarvis~A.~Schultz and  Todd~D.~Murphey are with the Department of Mechanical Engineering, Northwestern University, 2145 Sheridan Road Evanston, IL 60208, USA 
        {\tt\small Email: stacymav@u.northwestern.edu; jschultz@northwestern.edu; t-murphey@northwestern.edu}}}

\maketitle

\begin{abstract}

This paper considers the problem of real-time mode scheduling in linear time-varying switched systems subject to a quadratic cost functional. The execution time of hybrid control algorithms is often prohibitive for real-time applications and typically may only be reduced at the expense of approximation accuracy.
We address this trade-off by taking advantage of system linearity to formulate a projection-based approach so that no simulation is required during open-loop optimization. A numerical example shows how the proposed open-loop algorithm outperforms methods employing common numerical integration techniques. 
Additionally, we follow a receding-horizon scheme to apply real-time closed-loop hybrid control to a customized experimental setup, using the Robot Operating System (ROS).  
In particular, we demonstrate---both in Monte-Carlo simulation and in experiment---that optimal hybrid control efficiently regulates a cart and suspended mass system in real time. 
 

\end{abstract}


\section{INTRODUCTION}
This paper considers the problem of mode scheduling---selection of both sequence and timing of modes---for an autonomous linear time-varying switched system to optimize a quadratic performance metric. The contribution of this paper is two-fold: 
1) The formulation of an open-loop mode scheduling algorithm (referred to as Single Integration Optimal Mode Scheduling---SIOMS) so that no differential equation needs to be solved for during optimization; 
2) The formulation and experimental implementation of a receding-horizon mode scheduling algorithm for real-time closed-loop hybrid control so that a differential equation only needs to be integrated over a limited time interval $\delta$---typically the time step of the receding-horizon window---rather than the time horizon $T$.

Optimal scheduling problems arise in a number
of application domains, such as mobile robotics \cite{boccadoro2005optimal,caldwell2011switching}, hybrid automotive control \cite{lygeros1998verified,uthaichana2011hybrid,meyer2011hybrid}, power electronics \cite{mariethoz2010comparison}, telecommunications \cite{hristu2001feedback,bhattacharya2011switching} and air traffic management \cite{sastry1995hybrid, GNC11}. 
Several algorithmic (theoretical and practical) methods have been proposed to deal with these problems.  Mode scheduling is challenging due to the fact that both the mode sequence and the set of switching times are optimized jointly. As a result, many approaches have relied on a bi-level hierarchical structure with only a subset of the design variables considered at each level \cite{axelsson2008gradient,egerstedt2006transition,gonzalez2010numerical,gonzalez2010descent}. Other proposed methods include: embedding methods \cite{bengea2005optimal,wei2007applications}, which relax, or embed, the integer constraint and find the optimal of the relaxed cost;
relaxed dynamic programming \cite{rantzer2006relaxed,gorges2011optimal} where complexity is reduced by relaxing optimality within pre-specified bounds;  and variants of gradient-descent methods \cite{wardi2014switched,wardi2012algorithm,passenberg2010minimum}.

The iterative projection-based approach as introduced in \cite{838,839,840}  forms the basis for the work in this paper. The mode scheduling problem is formulated as an infinite-dimensional optimal control problem where the variables to be optimized are a set of functions of time constrained to the integers. For a projection-based method, the design variables are in an unconstrained space but the cost is computed on the projection of the design variables to the set of admissible switched system  trajectories. In \cite{838},  an iterative optimization algorithm is synthesized that employs the Pontryagin Maximum Principle and a projection-based technique. 
We adapt this algorithm by taking advantage of the linearity of the dynamical system under concern. The specific case of linear switching control has been extensively investigated by others (see \cite{sun2005analysis,lincoln2001optimizing,seatzu2006optimal,zhang2008optimal,kouhi2013suboptimality,SCL10,HSCC08, egerstedt2000toward}).  
The approaches in \cite{giua2001optimal,xu2001approach} solve for a differential equation at each step of the iterative algorithm. Previous attempts to avoid the on-line integration of differential equations are limited to switching time optimization problems where the mode sequence is fixed \cite{giua2001optimal,xu2001approach,xu2002optimal,caldwell2012single}.

In this paper, we extend our work in \cite{mavrommati2014single} to present and experimentally evaluate  a projection-based mode scheduling algorithm (SIOMS) where  a single set of differential equations is solved off-line, so that no additional simulation is required during the open-loop optimization routine. \textit{These off-line solutions to differential equations are independent of the mode sequence and switching times} in contrast to \cite{giua2001optimal,xu2001approach,xu2002optimal,caldwell2012single}. Moreover, no assumption about the time-variance of the modes is made. Therefore, SIOMS does not exclude many important linear systems, such as time-varying power systems \cite{goldsmith1997variable,neely2006energy} and nonlinear systems linearized about a trajectory.

One of the strongest assets of the proposed algorithm is that its \textit{timing behavior--- i.e., the execution time of a single iteration of the optimization algorithm--- is independent of any choice of ODE solver}, and only depends on the number of multiplications and inversions required for the calculation of the optimality condition. That is, although the execution times of iterative hybrid control methods are often prohibitive for real-time implementation  \cite{vasudevan2013consistent1,vasudevan2013consistent2}, \textit{SIOMS is fast} and intrinsically free of the common trade-off between fast execution and low approximation errors that normally determines the selection of the appropriate numerical integration technique.

Furthermore, by avoiding online simulations, our proposed formulation provides a solution for implementation of hybrid control algorithms in real-time applications. In particular, the numerical implementation of switched system algorithms that require the exact solution of differential equations may be impractical for two reasons. First, the solution approximation through discretization does not always guarantee consistency \cite{polakoptimization, sanfelice2010dynamical}; second, discontinuous differential equations require specialized event-based numerical techniques that are prone to approximation errors \cite{brogliato2008numerical}. Authors in \cite{vasudevan2013consistent2} address the first issue and propose a method to apply a discretization that guarantees consistency, considering nonlinear systems. In this paper, by restricting our focus to linear time-varying systems, we introduce a method where \textit{approximation accuracy and consistency are independent of the number of samples} used for approximation of the state and co-state trajectories. 
To address the second issue, SIOMS only requires off-line integration of differential equations that are continuous and as smooth as each of the linear modes. Thus, our algorithm exhibits \textit{robustness to numerical errors due to discontinuous vector fields}.

As a result of the aforementioned computational and timing advantages, SIOMS is particularly suited for closed-loop infinite horizon control by means of a receding-horizon synthesis \cite{mayne2014model}. Importantly, in receding-horizon optimization using SIOMS, a simple update step removes the need for numerical integration over the full time horizon $T$ in between consecutive algorithm runs---only integration over a few time steps $dt$ is required. Benefits of this paper's method include improved robustness to model uncertainties and disturbance rejection capabilities as demonstrated by the experimental results and a Monte-Carlo analysis. In particular,  we use the Robot Operating System (ROS) to apply real-time hybrid control to a mobile robot and suspended mass system. Our numerical and experimental work demonstrates that closed-loop SIOMS  regulates the example system reliably in real time.

This paper is structured as follows: Section \ref{review} reviews switched systems and their representations while stating the optimization problem. The single-integration mode scheduling algorithm is proposed in Section \ref{algor} where a receding-horizon approach for closed-loop control is also introduced. Details about the numerical implementation of open-loop SIOMS are provided in Section \ref{open}, along with a comparison with former implementations \cite{838, 839}. Finally, Section~\ref{closed} illustrates the efficacy of closed-loop SIOMS through a Monte-Carlo analysis and a real-time experiment.

\section{ Review}
\label{review}
\subsection{Switched Systems}
\label{def}

Switched systems are a class of hybrid systems \cite{goebel2009hybrid, lygeros2005overview} that evolve according to one of $N$ vector fields (modes) $f_i: \mathbb{R}^n \rightarrow \mathbb{R}, i \in \{1,...,N\}$ at any time over the finite time interval $[T_0,T_M]$, where $T_0$ is the initial time and $T_M>0$ is the final time. We consider two representations of the switched system, namely mode schedule and switching control. As a unique mapping exists between each representation \cite{838}, the two will be used interchangeably throughout the paper. 

\textbf{Definition 1} The \textit{mode schedule}  is defined as the pair $\{\varSigma,\mathcal{T}\}$ where $\varSigma=\{\sigma_1,...,\sigma_M\}$ is the sequence of active modes $\sigma_i \in \{1,...,N\}$ and $\mathcal{T}=\{T_1,...,T_{M-1}\}$ is the set of the switching times $T_i\in[T_0,T_M]$. The total number of modes in the mode sequence---which may vary across optimization iterations---is $M\in\mathbb{Z}^+$.

\textbf{Definition 2} A \textit{switching control}  corresponds to a list of curves $u=[u_1,...,u_N]^T$ composed of $N$ piecewise constant functions of time, one for each different mode $f_i$. For all $t\in[T_0,T_M], \sum\nolimits_{i=1}^N u_i(t)=1$, and for all $i\in\{1,...,N\}, u_i(t)\in\{0,1\}$. This dictates that the state evolves according to only one mode for all time. We represent the set of all admissible switching controls as $\Omega$. 

We will refer to the mode schedule corresponding to the switching control $u$ as $\{\varSigma(u),\mathcal{T}(u)\}$.

For a system with $n$ states $x=[x_1,...,x_n]^T$ and $N$ different modes, the state equations are given by
\begin{equation}
\label{eq15}
\dot{x}(t)=F(t,x(t),u(t)):= \sum\limits_{i=1}^{N} u_i(t) f_i(x(t),t) 
\end{equation}
subject to the initial condition $x(T_0)=x_0$.
For this paper, we restrict our focus to linear time-varying systems so that 
\begin{equation}
\label{eq17}
F(t,x(t),u(t)):= \sum\limits_{i=1}^{N} u_i(t) A_i(t) x(t).
\end{equation}
Alternatively, we may express the system dynamics with respect to the current mode schedule as follows: 
\begin{equation}
\label{eq16}
F(t,x(t),\varSigma,\mathcal{T}):= \overline{A}(t,\varSigma,\mathcal{T}) x(t)
\end{equation}
where $\overline{A}(t,\varSigma,\mathcal{T})=A_{\sigma_i}(t) $ for $ T_{i-1}\leq t<T_i$.

\subsection{Problem Statement} 
\label{problem}
Our objective is the minimization of a quadratic cost function 
\begin{equation}
\label{cost}
J(x,u)=  \int_{T_0}^{T_M} \frac{1}{2} x(\tau)^T Q(\tau) x(\tau)d\tau + \frac{1}{2} x(T_M)^T P_1 x(T_M)
\end{equation}
subject to the pair $(x,u)$ where $x$ is the state and $u$ the switching control. Here, $Q$ and $P_1$ are the running and terminal cost respectively, and are both symmetric positive semi-definite. Note that this cost functional can also be adapted to include reference trajectory, in which case the objective would be to minimize the error between the state and the reference (\cite{caldwell2012single}). Trajectories $(x,u)$ that optimize the performance metric (\ref{cost}) are constrained by the state equations \eqref{eq17}.  

\subsection{Projection-based Optimization} 
From Definition 2 of an admissible switching control $u$, it follows that our optimization problem is subject to an integer constraint \cite{838}. Let $\mathcal{S}$ represent the set of all pairs of admissible state and switching control trajectories $(x,u)$, i.e.\ all pairs that satisfy the constraint (\ref{eq15}) and are consistent with Definition 2 so that $u\in \Omega$. In \cite{839}, the authors propose a projection-based technique for handling these constraints set by $\mathcal{S}$. In particular, an equivalent problem is considered where the design variables $(\alpha,\mu)$ belong to an unconstrained set $(\mathcal{X},\mathcal{U})$ and the cost $J$ is evaluated on the projection of these variables to the set $\mathcal{S}$. 
Now, the problem is reformulated as
\begin{equation}
\label{costproj}
\arg \min\limits_{(\alpha,\mu)}J(\mathcal{P}(\alpha,\mu))
\end{equation}
where $\mathcal{P}$ is a projection - with  $\mathcal{P}(\mathcal{P}(\alpha,\mu))  = (\mathcal{P}(\alpha,\mu))$ - that maps curves from the unconstrained set $(\mathcal{X},\mathcal{U})$ to the set of admissible switched systems $\mathcal{S}$. As the cost is calculated on the admissible projected trajectories, this problem is equivalent to the original problem described in  \ref{problem} and \eqref{cost}.    

The optimal mode scheduling algorithm developed in \cite{838} utilizes the max-projection operator. The max-projection operator $\mathcal{P}: \mathcal{X}\times \mathcal{U}  \rightarrow \mathcal{S} $ at time $t\in [T_0,T_M]$ is defined as 
\begin{equation}
\label{proj}
\mathcal{P}(\alpha(t),\mu(t)):= 
\begin{cases}
&\dot{x}(t)=F(t,x(t),u(t)), \;\;\;\;\;\;\; x(T_0)=x_0\\
& u(t) = \mathcal{Q}(\mu (t))
\end{cases}
\end{equation}
where $\mathcal{Q}$ is a mapping from a list of $N$ real-valued control trajectories, $\mu(\cdot)=$ $[\mu_1(\cdot),...,\mu_N(\cdot)]^T$ $\in \mathbb{R}^N$ to a list of $N$ feasible switching controls, $u\in \Omega$. We define $\mathcal{Q}$ as
\begin{equation}
\label{pro}
\mathcal{Q}(\mu (t)) = \begin{bmatrix}
\mathcal{Q}_1(\mu (t)) \\
\vdots\\
\mathcal{Q}_N(\mu (t)) 
\end{bmatrix}\;\; \text{with} \;\;\; 
\mathcal{Q}_i(\mu (t)):= \prod\limits_{j\neq i }^{N} 1(\mu_i (t)- \mu_j (t))
\end{equation} 
where $1:\mathbb{R}\rightarrow \{0,1\}$ is the step function given by
\begin{equation}
\label{step}
1(t) = 
 \begin{cases}
1, & t\geq0 \\
0, & else. 
\end{cases}
\end{equation} 
Notice that the max-projection operator does not depend on the unconstrained state trajectories $\alpha(\cdot)$. The unconstrained state $\alpha$ is included in the left hand side of the
definition in order for $\mathcal{P}$ to be a projection. 
 
\subsection{Mode Insertion Gradient}
\label{grad}
The \textit{mode insertion gradient} appears in previous studies \cite{egerstedt2006transition,gonzalez2010descent,wardi2012algorithm}. Here, it is defined as the list of functions $d = [d_1(t),...,d_N(t)] \in \mathbb{R}^N$ that calculate the sensitivity of cost $J$ to inserting one of the $N$ modes at some time $t$ for an infinitesimal interval (i.e. $\frac{dJ}{d\lambda^+}$ as $\lambda^+\rightarrow 0$). Each element of $d$ is given by:   
\begin{equation}
\label{mode}
d_i(t): = \rho(t)^T (f_i(x(t),t) - f_{\sigma(t)}(x(t),t)) , \;\; i=1,...,N
\end{equation}
where $x\in \mathbb{R}^n$ is the solution to the state equations (\ref{eq15}) for all $t\in[T_0,T_M]$ and  $\rho \in \mathbb{R}^n$, the co-state, is the solution to the adjoint equation 
\begin{equation}
\label{rhoeq}
\dot{\rho}(t)= - \mathcal{D}_x F(t,x(t),u(t))^T \rho(t)-Q(t)x(t), 
\end{equation}  
for all $t\in[T_0,T_M]$ subject to $\rho(T_M)=P_1x(T_M)$.
(In (\ref{mode}), $\sigma(t): [T_0,T_M]\rightarrow \{1,...,N\}$ is the function that returns the active mode at any time $t$.)  

It has been shown in \cite{841},  that when a quadratic cost is optimized subject to a linear time-varying switched system, a linear mapping between state $x$ and co-state $\rho$ exists. Thus, we may express the co-state as 
\begin{equation}
\label{rhoeq1}
\rho(t)=P(t) x(t)
\end{equation} 
where $P(t)\in \mathbb{R}^{n\times n}$ is calculated by the following differential equation:
\begin{equation}
\label{Rhoeq}
\dot{P}(t)= -\overline{A}(t,\varSigma,\mathcal{T})^TP(t)-P(t)\overline{A}(t,\varSigma,\mathcal{T})-Q(t) 
\end{equation} 
subject to $P(T_M)=P_1$.
Note that this is the linear switched system analog to the Riccati equation from the LQR problem in classical control theory  \cite{anderson2007optimal}. 
Using (\ref{eq17}) and (\ref{rhoeq1}), the mode insertion gradient element can be written as
\begin{equation}
\label{modenew}
d_i(t): = x(t)^T P(t)^T [A_{i}(t) - A_{\sigma (t)}(t)]x(t).
\end{equation}

\subsection{Iterative Optimization}

To calculate the switching control $u(t)$ that optimizes the quadratic performance metric \eqref{cost}, we follow an iterative approach. Iterative optimization computes a new estimate of the optimum by taking a step from the current estimate in a search direction so that a sufficient decrease in cost is achieved \cite{axelsson2008gradient, wardi2014switched, gonzalez2010descent,838}.
A single iteration is commonly structured in the following scheme: Given a current estimate of the optimum, i) Calculate a descent direction; ii) Calculate a step size; iii) Update the current estimate by taking a step in the descent direction. The iterative procedure is repeated until a terminating condition is satisfied.
The descent must be sufficient so that the sequence generated by the iterative optimization algorithm converges to a stationarity point---or at least sufficient for an optimality function to go to zero \cite{wright1999numerical,armijo1966minimization,840,wardi2014switched,axelsson2008gradient}. 

In the following section, we formulate an iterative projection-based algorithm for quadratic optimization of linear time-varying switched systems that requires no online simulations.

\section{Single Integration Optimal Mode Scheduling}
\label{algor}

\begin{algorithm}
	\footnotesize{
     \caption{ SIOMS}
     \label{sioms} 
     \hrule
     
     \vspace{2 mm}
     \textit{Off-line:}
      
      \begin{itemize}
        \item Solve for the STM $\Phi^j(t,T_0)$ and ATM $\Psi^j(t,T_M)$
                          $ \forall j \in \{1,...,N\} $ and $t\in[T_0,T_M]$.
                      \item Choose initial $u^0\rightarrow \{\varSigma(u^0),\mathcal{T}(u^0)\}$.
                      \item Set $x(T_0)=x_0$ and $P(T_M)=P_1$.
        
      \end{itemize}
      
      \hrule
      \vspace{2 mm}
      \textit{On-line iterative process:}
      
      Set $k=0$, $u^k=u^0$.  
      \begin{enumerate}
              \item Evaluate $x^k(t):=\chi(t,\varSigma(u^k),\mathcal{T}(u^k))$ as in Eq. (\ref{eq2}). 
              \item Evaluate $P^k(t):=\varrho(t,\varSigma(u^k),\mathcal{T}(u^k))$ as in Eq. (\ref{rho}).
              \item Evaluate the descent direction $-d^k(t)$ as in Eq. (\ref{modedef}).
              \item Calculate step size $\gamma^k$ by backtracking.
              \item Update: $u^{k+1}(t)=Q(u^k(t)-\gamma^k d^k(t)) $.
              \item If  $u^{k+1}$ satisfies a terminating condition, then exit, else, increment $k$ and repeat from step 1. 
              
      \end{enumerate}
      }
\end{algorithm}

\subsection{ Open Loop Control over Finite Time Horizon } 
\label{algor1} 
The problem of optimizing an arbitrary cost functional $J(x,u)$ subject to the switching control $u(t)$ and switching system state $x(t)$ is considered in  \cite{838}. Here, we extend  \cite{838} by restricting to linear time-varying systems with quadratic performance metric. In particular, we  reformulate this problem so that no differential equations are solved during the iterative optimization routine. Algorithm \ref{sioms} provides a summary of SIOMS.

Consider the optimization problem constrained by the system dynamics \eqref{eq16}, as described in section \ref{review}. The dynamic constraint dictates that a system simulation should be performed at each iteration in Algorithm \ref{sioms} as soon as the next switching control has been calculated. In particular, the calculation of the mode insertion gradient (\ref{mode}) involves the solution of the state and adjoint equations, (\ref{eq16}) and (\ref{rhoeq}), while the max-projection operator also includes the state equation \eqref{eq16}.  

We follow a similar approach to the switching time optimization approach in \cite{caldwell2012single}, extending it to the situation where the mode sequence $\varSigma$ is unknown. Building on the existence of a linear relationship between the state and co-state as described in section \ref{grad}, we  utilize operators to formulate algebraic expressions for the calculation of the state $x(t)$ and the relation $P(t)$ at any time $t\in[T_0,T_M]$. The operators are available prior to optimization through off-line solutions to differential equations. Moreover, they are independent of the mode sequence and switching times. 

In the switching time optimization case \cite{caldwell2012single}---where  the mode sequence is constant and the problem is finite-dimensional---a single optimization iteration involves only a finite number of state and co-state evaluations; these occur at the (finite) switching times for that particular iteration. However, mode scheduling is an infinite-dimensional optimal control problem and requires  the time evolution of the state and co-state trajectories at each iteration.
Therefore, in order for the proposed algorithm to be feasible, an explicit mapping from time $t$ to $x$ and $P$ is needed at each iteration, depending on the current mode schedule $\{\varSigma,\mathcal{T}\}$. The mapping, below in (\ref{eq2}) and (\ref{rho}), only includes algebraic expressions dependent on solutions to pre-computed differential equations. The exact number of multiplications executed in each iteration depends on how many time instances the state and co-state must be evaluated. 
 
For the rest of the paper, a variable with the superscript $k$ implies that the variable depends directly on $u^k$ i.e.\ the switching control at the $k$\textsuperscript{th} algorithm iteration.

\paragraph{Evaluating $x(t)$}
The operators for evaluating $x(t)$ are the state-transition matrices (STM)  of the $N$ different modes. Let $\Phi^j(\cdotp,T_0): \mathbb{R}  \rightarrow \mathbb{R}^{n\times n} $ denote the STM for the linear mode $j\in\{1,...,N\}$ with $A_j(t)$. 
The STM are the solutions to the $N$ matrix differential equations 
\begin{equation}
\label{eq55}
\frac{d}{dt}\Phi^j(t,T_0)=A_j(t) \cdotp\Phi^j(t,T_0),\;\; j=1,...,N 
\end{equation}
subject to the initial condition $ \Phi^j(T_0,T_0)=I_n$.
The following two STM properties are useful for computing the state $x(t)$ given a mode schedule $\{\varSigma,\mathcal{T}\}$. For an arbitrary STM, $\Phi$, characterized by $A(t) $, we have \cite{hespanha2009linear} :
  \begin{enumerate}
              \item $x(t)=\Phi(t,\tau) x(\tau)  $
              \item $\Phi(t_1,t_3)= \Phi(t_1,t_2)\Phi(t_2,t_3)= \Phi(t_1,t_2)\Phi(t_3,t_2)^{-1}. $
              
  \end{enumerate}
We emphasize the importance of Property 2 in that it allows us to use a single operator for the evaluation of the state as explained in the following.

\begin{proposition}
 The state $x(t)$ at all $t\in[T_0,T_M]$ depends on the mode schedule $\{\varSigma,\mathcal{T}\}$ and the STM $\Phi^j(\cdotp,T_0)$ and is given by $x(t):=\chi(t,\varSigma,\mathcal{T})$ where  
\begin{equation}
\label{eq2}
\begin{split}
 \chi(t,\varSigma,\mathcal{T})= \sum_{i=1}^{M} \bigg\{ &\Big[1(t-T_{i-1})-1(t-T_i)\Big]\\&\Phi^{\sigma_i}(t,T_0)\Phi^{\sigma_i}(T_{i-1},T_0)^{-1}x(T_{i-1})\bigg\}\\
& \mbox{subject to } x(T_0)=x_0,
\end{split}
\end{equation}
$1(\cdot)$ is the step function defined in \eqref{step} and  $T_{i}$ , $\sigma_i$ are the $i^{th}$ switching time and corresponding active mode as defined in Section \ref{def}. 
\end{proposition} 

\begin{proof}

Using the STM properties 1 and 2, the state $x$ at the $i$\textsuperscript{th} switching time is
\begin{equation}
\label{eq3}
x(T_i)=\overline{\Phi}(T_i,T_0)x_0=\left[\prod_{j=i}^1 \Phi^{\sigma_j}(T_j,T_{j-1}) \right]x_0
\end{equation}
where $\overline{\Phi}(T_i,T_0)$ is the state-transition matrix corresponding to  $\overline{A}(t,\varSigma,\mathcal{T})$ as defined in \eqref{eq16}. 
Hence, the state evolution is defined as a piecewise function of time, each piece corresponding to a time interval between consecutive switching times $\{T_i,T_{i+1}\}$: 
\begin{equation}
\label{eq9}
 x(t) = 
\begin{cases}
\Phi^{\sigma_1}(t,T_{0})x(T_0), & T_0\leq t<T_1 \\
\Phi^{\sigma_2}(t,T_{1})\Phi^{\sigma_1}(T_1,T_0)x(T_0), & T_1\leq t <T_2 \\
\;\;\;\;\;\;\;\;\;\; \vdots & \;\;\;\;\;\;\;\;\vdots \\
\Phi^{\sigma_M}(t,T_{M-1}) [\prod\limits_{j=M-1}^1 \Phi^{\sigma_j}(T_j,T_{j-1})]x(T_0)     & T_{M-1}\leq t \leq T_M
\end{cases}
\end{equation}
For a more compact representation of the state, we employ unit step functions and (\ref{eq3}) to get 
\begin{equation}
\label{eq132}
x(t)= \sum_{i=1}^{M}  \Big\{ [1(t-T_{i-1})-1(t-T_i)]\Phi^{\sigma_i}(t,T_{i-1})x(T_{i-1})\Big\}
\end{equation}
where, from STM property 2,  
\begin{equation}
\label{eq7}
\Phi^{\sigma_i}(t,T_{i-1})= \Phi^{\sigma_i}(t,T_0)\Phi^{\sigma_i}(T_{i-1},T_0)^{-1}.
\end{equation}
This concludes the proof. 
\end{proof}

Prior to the iterative optimization, the STM operators $\Phi^j(t,T_0)$ are solved off-line for $t\in[T_0,T_M]$ and for all different modes $j=1,...,N$. Thus, given a mode schedule, the calculation of state $x(t)$ via (\ref{eq2})  requires no additional integrations beyond the off-line calculations used for $\Phi^j(t,T_0)$.

\paragraph{Evaluating $P(t)$}
As proven in \cite{caldwell2012single}, an analogous operator to the STM exists for the evaluation of the relation $P(t)$ appearing in (\ref{rhoeq1}). As in \cite{caldwell2012single}, we will refer to the operator as the adjoint-transition matrix (ATM) and use $\Psi^j(\cdotp,T_M):\mathbb{R}  \rightarrow \mathbb{R}^{n\times n} $ to denote the ATM corresponding to each mode $j\in\{1,...,N\}$. The ATM are defined to be the solutions to the following $N$ matrix differential equations:
\begin{equation}
\label{eq8}
\frac{d}{dt}\Psi^j(t,T_M)= -A_j(t)^T \Psi^j(t,T_M)-\Psi^j(t,T_M)A_j(t)-Q(t)
\end{equation}
subject to the initial condition $\Psi^j(T_M,T_M)=0_{n\times n}$.

The following two ATM properties will be useful for evaluating $P(t)$ given a mode schedule $\{\varSigma,\mathcal{T}\}$. For an arbitrary ATM, $\Psi$, characterized by $A(t) $ and associated STM $\Phi$, and cost function defined by $Q(t)$, we have \cite{caldwell2012single}:
       \begin{enumerate}
              \item $P(t)=\Psi(t,\tau)\circ P(\tau) := \Psi(t,\tau)+\Phi(\tau,t)^T P(\tau)\Phi(\tau,t) $
              \item $\Psi(t_1,t_3)= \Psi(t_1,t_2)\circ \Psi(t_2,t_3):= \Psi(t_1,t_2)+\Phi(t_2,t_1)^{T}\Psi(t_2,t_3)\Phi(t_2,t_1). $              
      \end{enumerate}
Notice that Property 2 of ATM is equivalent to Property 2 of STM and similarly allows us to evaluate the co-state.  

\begin{proposition}
The relation $P(t)$ at all $t\in[T_0,T_M]$ depends on the current mode schedule $\{\varSigma,\mathcal{T}\}$, the STM $\Phi^j(\cdotp,T_0)$  and the ATM $\Psi^j(\cdotp,T_M)$ and is given by  $P(t):=\varrho(t,\varSigma,\mathcal{T})$ where
\begin{equation}
\label{rho}
\begin{split}
\varrho(t,\varSigma,\mathcal{T})
= \sum_{i=1}^{M} \bigg\{ \Big[1(t-T_{i-1})-1(t-T_i)\Big]&\cdotp\\ \Big[\Psi^{\sigma_i}(t,T_M)+ \Phi^{\sigma_i}(t,T_0)^{-T}   \Phi^{\sigma_i}(T_i,T_0)^{T}   [P(T_{i})&\\-\Psi^{\sigma_i}(T_{i},T_M)] \Phi^{\sigma_i}(T_i,T_0) & \Phi^{\sigma_i}(t,T_0)^{-1} \Big] \bigg\}\\
 \mbox{subject to } P(T_M)=& P_1,  
\end{split}
\end{equation}
$1(\cdot)$ is the step function defined in \eqref{step} and  $T_{i}$ , $\sigma_i$ are the $i^{th}$ switching time and corresponding active mode as defined in Section \ref{def}. 
\end{proposition} 

\begin{proof}
From the ATM properties 1 and 2, $P(t)$ at the $i$\textsuperscript{th} switching time is
\begin{equation}
\begin{split}
P(T_i)&=\overline{\Psi}(T_i,T_M)\circ  P(T_M)\\&=\overline{\Psi}(T_i,T_M)+\overline{\Phi}(T_M,T_i)^T P(T_M)\overline{\Phi}(T_M,T_i)
\end{split}
\end{equation}
where $\overline{\Psi}(T_i,T_M)$ is the adjoint-transition matrix corresponding to $\overline{A}(t,\varSigma,\mathcal{T})$ as defined above. From ATM property 2, this is equal to
\begin{equation}
\label{eq10}
\begin{split}
\overline{\Psi}(T_i,T_M)=&{\Psi}^{\sigma_{i+1}}(T_i,T_{i+1})\circ \cdots \circ {\Psi}^{\sigma_{M}}(T_{M-1},T_{M}) \\
=& \sum\limits_{m=i+1}^{M} \overline{\Phi}(T_{m-1},T_i)^T {\Psi}^{\sigma_{m}}(T_{m-1},T_{m}) \overline{\Phi}(T_{m-1},T_i).
\end{split}
\end{equation}
As in the previous case, we aim to derive an expression for the evaluation of $P(t)$ at random time instances, as needed. Again, we will represent $P(t)$ as a piecewise function of time: 
\begin{equation}
\label{eq12}
 P(t) = 
 \begin{cases}
\Psi^{\sigma_M}(t,T_{M})\circ P(T_M), & T_{M-1}\leq t<T_M \\
\Psi^{\sigma_{M-1}}(t,T_{M-1})\circ  P(T_{M-1}), & T_{M-2}\leq t <T_{M-1} \\
\;\;\;\;\;\;\;\;\;\; \vdots & \;\;\;\;\;\;\;\;\vdots \\
\Psi^{\sigma_{1}}(t,T_{1})\circ  P(T_1), & T_{0}\leq t <T_{1}
\end{cases}
\end{equation}
For a more compact representation of $P(t)$, we employ unit step functions to get
\begin{equation}
\label{rho2}
\begin{split}
P(t)=& \sum_{i=1}^{M} \Big\{ [1(t-T_{i-1})-1(t-T_i)][\Psi^{\sigma_i}(t,T_{i})\circ P(T_{i})] \Big\}\\
\end{split}
\end{equation} 
where, from ATM property 2,
\begin{equation}
\label{psi}
\Psi^{\sigma_i}(t,T_{i})= \Psi^{\sigma_i}(t,T_M)+\Phi^{\sigma_i}(T_i,t)^{T}\Psi^{\sigma_i}(T_i,T_M)\Phi^{\sigma_i}(T_i,t).
\end{equation}
Combining ATM property 1 with (\ref{rho}) and (\ref{psi}), we end up with the expression
\begin{equation}
\label{rho3}
\begin{split}
P(t)&
= \sum_{i=1}^{M} \Big\{ [1(t-T_{i-1})-1(t-T_i)]\cdotp\\& [\Psi^{\sigma_i}(t,T_M)+ \Phi^{\sigma_i}(T_{i},t)^T [P(T_{i})-\Psi^{\sigma_i}(T_{i},T_M)] \Phi^{\sigma_i}(T_{i},t)] \Big\} 
\end{split}
\end{equation}
with $ \Phi^{\sigma_i}(T_{i},t) = \Phi^{\sigma_i}(T_i,T_0)  \Phi^{\sigma_i}(t,T_0)^{-1}$.
This completes the proof. 
\end{proof}

Prior to the iterative optimization, the ATM operators $\Psi^j(t,T_M)$ are solved off-line for all $t\in[T_0,T_M]$ and for all different modes $j=1,...,N$. Thus,
given a mode schedule, the calculation of $P(t)$ via (\ref{rho})  requires no additional integrations.

\paragraph{Calculating the descent direction using the mode insertion gradient}
An iterative optimization method computes a new estimate of the optimum by taking a step in a search direction from the current estimate of the optimum so that a sufficient decrease in cost is achieved. The mode insertion gradient $d(t)$ defined above, has a similar role in the mode scheduling optimization as the gradient does for finite-dimensional optimization. It has been shown in \cite{838,egerstedt2006transition,gonzalez2010descent} that $-d^k(t)$ is a descent direction.  

\begin{proposition}
An element of $d^k(t)$ is given by
\begin{equation}
\label{modedef}
\begin{split}
d_i^k(t): = \chi(t,\varSigma(u^k),\mathcal{T}(u^k))^T& \varrho(t,\varSigma(u^k),\mathcal{T}(u^k))\\&[A_{i}(t) - A_{\sigma^k (t)}(t)]\chi(t,\varSigma(u^k),\mathcal{T}(u^k))
\end{split}
\end{equation}
where $i=\{1,...,N\}$.
\end{proposition} 

\begin{proof}
After the definition for the state and co-state, an equivalent expression for the mode insertion gradient may be obtained from (\ref{eq2}),(\ref{rho})
and (\ref{mode}). 
\end{proof}

\paragraph{Update rule} 
A new estimate of the optimal switching control $u^{k+1}$ is obtained by varying from the current iterate $u^k$ in the descent direction and projecting the result to the set of admissible switching control trajectories. For this purpose, we employ  the max-projection operator (\ref{proj}) and get a new estimate of the optimum,
\begin{equation}
\label{projn}
\begin{split}
& u^{k+1}(t) = \mathcal{Q}(u^k(t)-\gamma^k d^k(t))\\
& x^{k+1}(t):=\chi(t,\varSigma(u^{k+1}),\mathcal{T}(u^{k+1}))
\end{split}
\end{equation}
where $\mathcal{Q}$ is given by (\ref{pro}). For choosing a sufficient step size $\gamma ^k$, we may utilize a projection-based backtracking process as described in \cite{840}.
 
The reader is referred to \cite{838,839} for a more detailed description of these algorithm steps, along with the associated proofs for convergence. 

\paragraph{Calculating the optimality condition}
The optimality function $\theta^k \in \mathbb{R}$  is \cite{838}  
\begin{equation}
\label{opt}
\theta^k := d_{i_0}^k(t_0) 
\end{equation}
where
\begin{equation}
\label{opt1}
(i_0,t_0)= \arg \min\limits_{i\in \{1,...,N\}, t\in [T_0,T_M]}d_{i}(t). 
\end{equation}
The limit of the sequence of optimality functions is proven to go to zero as a function of iteration $k$ in  \cite{838}. This allows us to utilize $\theta^k$ also as a terminating condition for the iterative algorithm.

\subsection{A Receding-Horizon Approach}

\begin{figure}[b!]
\centering
\includegraphics{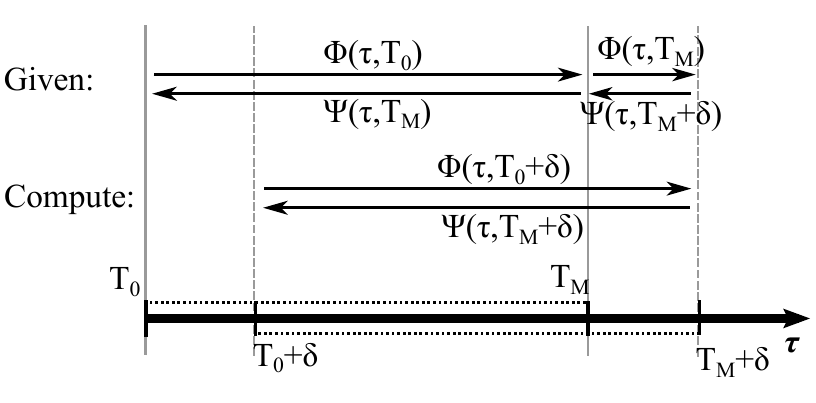} 
      \caption{ An illustration of the operators update step in a receding-horizon scheme. A differential equation needs to be integrated only over a limited time interval $\delta$ rather than the time horizon $(T_M-T_0):=T$. 
      }
      \label{update}
  \end{figure} 
  
Section \ref{algor1}  provides an off-line approach for computing an open-loop optimizer for the problem in section \ref{problem}. Here, we follow a receding-horizon approach in order to achieve closed-loop optimization over an infinite time horizon.

Receding-horizon control strategies \cite{mayne2014model,summers2011multiresolution} have become quite popular recently, partly due to their robustness to model uncertainties or to sensor measurement noise. This paper's approach enables real-time closed-loop execution of  finite-horizon optimal control algorithms. Based on our performance evaluation in the next section, the finite-horizon SIOMS is well-suited for receding-horizon linear switched-system control because it is fast and accurate.

A receding-horizon scheme for optimal mode scheduling can be implemented as follows.
From the current time $t$ and measured state $x(t)$ as the initial condition in \eqref{eq15}, use SIOMS to obtain an optimal switching control $u_t(\tau)$ for $\tau\in [t,t+T]$. Apply the calculated control for time duration $\delta$ with $0<\delta\leq T$  to drive the system from $x(t)$ at time $t$ to $x(t+\delta)$. Set $t\leftarrow t+\delta$ and repeat.  This scheme requires execution of the optimal mode scheduling algorithm every $\delta$ seconds. 

Following Algorithm \ref{sioms}, SIOMS requires an off-line calculation of operators before the on-line iterative process is executed. However, in order for SIOMS to be efficient in a receding-horizon approach, it is not preferable to recalculate each STM and ATM every $\delta$ seconds for the next $(T_M-T_0):=T$ seconds. Instead, each STM and ATM of the previous time interval $[T_0,T_M]$ are updated for the new information on
$[T_M,T_M+\delta]$ only (Fig. \ref{update}). Such an approach is feasible because of the following lemma.

\begin{lemma}
\label{upd}
 Suppose $\Phi(t,T_0)$ and $\Psi(t,T_M)$ are known for all $t\in[T_0,T_M]$. Assuming also that $\Phi(t,T_M)$ and $\Psi(t,T_M')$ are known for all $t\in[T_M,T_M']$, the STM and ATM for the time interval $t\in[T_0',T_M']$ with $T_0<T_0'$ and $T_M<T_M'$ are given by 
 \begin{equation}
 \label{update1} 
 \Phi(t,T_0') = 
  \begin{cases}
 \Phi(t,T_0) \Phi(T_0',T_0)^{-1}, & T_0'\leq t<T_M \\
\Phi(t,T_M) \Phi(T_M,T_0) \Phi(T_0',T_0)^{-1}, & T_M\leq t\leq T_M' 
 \end{cases}
 \end{equation}
and
 \begin{equation}
 \label{update2} 
 \Psi(t,T_M') = 
  \begin{cases}
 \Psi(t,T_M) \circ \Psi(T_M,T_M'), & T_0'\leq t<T_M \\
\Psi(t,T_M'), & T_M\leq t\leq T_M'. 
 \end{cases}
 \end{equation} 
\end{lemma}

\begin{proof}
The proof of Lemma \ref{upd} is a straightforward consequence of STM property 2 and ATM properties 1 and 2 in Section \ref{algor1}. 
\end{proof}

Despite its simplicity, Lemma \ref{upd} is the key to efficient real-time execution of a receding-horizon hybrid control scheme. 
Using Lemma \ref{upd} with $T_0' = T_0+\delta$ and $T_M' = T_M+\delta$, we formulate Algorithm \ref{rec} that allows for real-time closed-loop SIOMS execution. The proposed formulation requires a numerical integration over the limited time interval $\delta$ rather than the full time horizon $T$ (step 3.1 in Algorithm \ref{rec}). A graphical representation of the operators update every $\delta$ seconds (step 3 in Algorithm \ref{rec}) is given in Fig. \ref{update}.

\begin{algorithm}
	 \footnotesize{
     \caption{Receding-Horizon SIOMS}
     \label{rec}  
     \hrule
     
     \vspace{2 mm}
     \begin{itemize}
     \item Initialize current time $t$, finite horizon $T$ and control duration $\delta$. 
     
     \item Solve for $\Phi^j(\tau,t)$ and $\Psi^j(\tau,t+T)$   $\forall j \in \{1,...,N\}$ and $\tau \in [t,t+T]$.
     
     \end{itemize}
      \hrule
      \vspace{2 mm}
            
     Do every $\delta$ seconds while control $u_t(\tau)$ is applied: 
      \begin{enumerate}
              \item [\textbf{1}.] Update $T_0\leftarrow t$, $T_M\leftarrow t+T$ and set $x(T_0)=x(t)$. 
              \item [\textbf{2}.] Run on-line part of Algorithm \ref{sioms} to get $u_t(\tau)$ for $\tau\in [t,t+\delta]$.
               \begin{enumerate}
              \item [\textbf{3}.1] Solve for $\Phi^j(\tau,T_M)$ and $\Psi^j(\tau,T_M+\delta)$ $\forall \tau \in [T_M,T_M+\delta]$. *
              \item [\textbf{3}.2]Get $\Phi^j(\tau,T_0+\delta)$ and $\Psi^j(\tau,T_M+\delta)$ $\forall j \in \{1,...,N\}$ and $\tau \in [T_0+\delta,T_M+\delta]$  from known $\Phi^j(\tau,T_0)$ and $\Psi^j(\tau,T_M)$ using Lemma \ref{upd}. *
              \item [\textbf{3}.3]Update $\Phi(\tau,T_0) \leftarrow \Phi(\tau,T_0+\delta)$ and $\Psi(\tau,T_M) \leftarrow\Psi(\tau,T_M+\delta)$. *
              \end{enumerate} 
              
      \end{enumerate}
      \hrule
      \vspace{1 mm}
      \footnotesize{* In a real-time application, step 3 can be executed at any time when processing requirements are low, thus without increasing the amount of time needed for calculation of control (i.e.\ steps 1-2). }}
      
\end{algorithm}

 \begin{figure}
 \centering
 \includegraphics{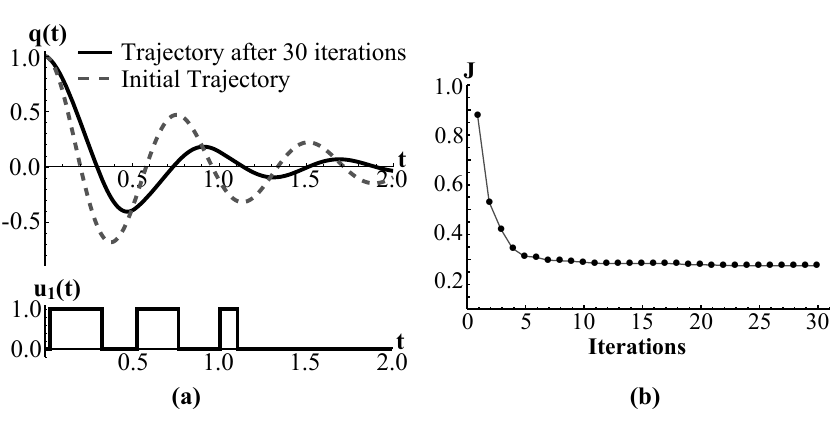} 
       \caption{ Spring-Mass-Damper vibration control: (a) Optimal trajectory and switching control and (b) the cost versus iteration count. 
       }
       \label{example1}
   \end{figure} 
  
\section{Open-Loop Implementation and  Evaluation}
\label{open}

In this section, SIOMS is implemented in a standard open-loop manner (see Algorithm \ref{sioms}) and its performance is evaluated in terms of i) execution time, ii) error of approximation and iii) computational complexity.  
 
As a baseline example, we use SIOMS to apply switched stiffness vibration control on an unforced spring-mass-damper system.  A linear time-invariant system  is particularly suited for evaluation purposes as an analytical solution exists and can be compared with the computed numerical solution. Variants of this example system have been used extensively in literature for the evaluation of hybrid controllers \cite{meyer2012comparison,cunefare2000state}.  Denoting by $k_i$ the variable spring stiffness and by $m$ and $d$ the mass and damping coefficient, the system equations take the form in \eqref{eq17} with 
\begin{equation}
\label{mass}
A_i(t)=  \begin{pmatrix}
& 0  & 1 &\\
& -\frac{k_i}{m} & -\frac{d}{m} &
\end{pmatrix}\;\;\;\;
\end{equation}
and $N=2$ i.e.\ two possible modes. The state vector is $x=[q(t),\dot{q}(t)]^T$, where $q(t)$ is the mass position. System parameters are defined as $m=1$, $d=2$, $k_1=30$, and $k_2=70$. Our objective is to find the mode schedule that minimizes the system vibration and is accordingly characterized by  the quadratic cost functional \eqref{cost} with $Q=diag[1,0.1]$, $P_1=0_{2\times2}$ and $[T_0,T_M]=[0,2]$. As an initial estimate $u^0(t)$, the system is in mode 2 with an initial condition $x_0=[1,0]^T$ and cost $J_0 \approx 0.98$. 

Fig. \ref{example1}a shows the optimal switching control and corresponding optimal $q(t)$ trajectory after 30 SIOMS iterations. The cost is reduced to $J \approx 0.38$ ( Fig. \ref{example1}b).

\subsection{Execution Time and Approximation Error}

\begin{figure}[t]
\centering
\includegraphics{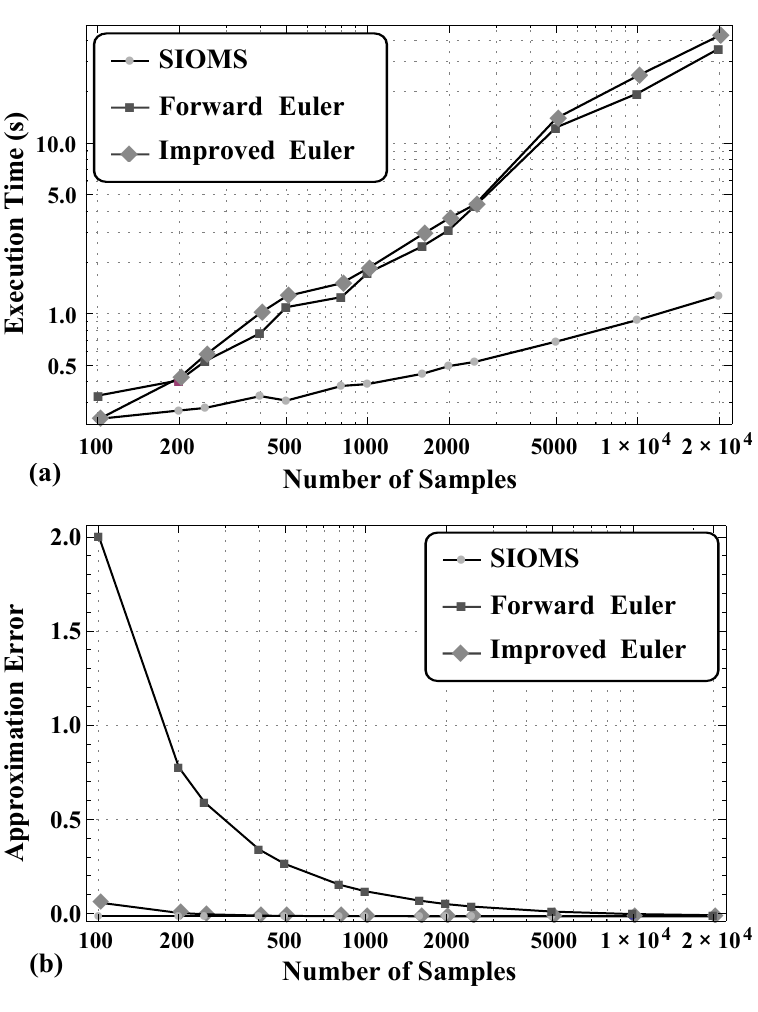} 
      \caption{ Variation of (a) execution times and (b) approximation errors (2-norm of the root-mean-squared differences between the analytic and computed state values) with respect to the selected number of samples evaluated across 3 different optimization methods. SIOMS can achieve both objectives (i.e.\ fast execution and high approximation accuracy) for a wide range of sample sizes. 
      }
      \label{dataplot}
  \end{figure}

The execution time of iterative optimal control algorithms might be prohibitive for real-time applications \cite{vasudevan2013consistent2}. It is often the case that choosing the appropriate numerical techniques for integrating the state and adjoint equations, \eqref{eq15} and \eqref{rhoeq}, is conducive to achieving lower execution times.  However, there is a trade-off to consider---a fast numerical ODE solver might be prone to approximation errors. 
In open-loop SIOMS (Algorithm \ref{sioms}), no differential equations need to be numerically solved as part of the online iterative process. Hence, we will show that both the execution time and approximation error can be kept low at the same time.  

Referring to Algorithm \ref{sioms}, a set of operator trajectories is pre-calculated and stored off-line, covering the full time horizon [$T_0,T_M$]. In practice, the exact number of stored samples $\mathcal{N}$ needs to be determined to reflect the processor's computational capacity and memory availability.\footnote{Interpolating methods may be used for intermediate time instances.} 
We use the mass-spring-damper to illustrate how the SIOMS execution time and error of approximation vary across different choices of sample sizes (Fig. \ref{dataplot}). 

For comparison purposes, we additionally evaluate the performance of the projection-based mode scheduling algorithm in \cite{838} using the same example and employing two different numerical techniques, namely the Forward and Improved Euler methods, for the integration of the state and adjoint equations. In contrast to SIOMS where expressions exist for the state and co-state evaluation (\eqref{eq2} and \eqref{rho}), here, the solution to the state and co-state equations, \eqref{eq17} and \eqref{rhoeq}, is approximated in every algorithm iteration. The Forward Euler method provides the following approximation to the state and co-state trajectories of a general linear time-varying switched system:
\footnotesize
\begin{equation}
\label{discrexpl}
\begin{split}
& x(t_{h+1})=(\mathbb{I} + \Delta t \cdot \overline{A}(t_h,\varSigma,\mathcal{T})) x(t_h), \;\;\;   \;\;\;  x(t_0)=x_0 \\
& \rho(t_h) = (\mathbb{I} + \Delta t \cdot \overline{A}(t_{h+1},\varSigma,\mathcal{T}))^T \rho(t_{h+1}) + \Delta t \cdot Q x(t_{h+1}),  \;\;\;  \;\;\;  \rho(t_{\mathcal{N}})=P_1 x(t_{\mathcal{N}}) 
\end{split}
\end{equation}
\normalsize
where $\mathbb{I}$ is the $n\times n$ identity matrix, $t_{h+1} = t_h + \Delta t $ with $\Delta t$ the step size and $\overline{A}(t,\varSigma,\mathcal{T})$ is defined in \eqref{eq16}. The Euler method is simple but can be unstable and inaccurate. On the other hand, Improved Euler (i.e.\ two-stage Runge Kutta) maintains simplicity but with reduced approximation errors. It applies the following approximation:
\small  
\begin{equation}
\label{discrimpr}
\begin{split}
 x(t_{h+1})=\Big[\mathbb{I} + \frac{\Delta t}{2}  A(t_h) + \frac{\Delta t}{2}  A(t_{h+1}) (\mathbb{I} + \Delta t \cdot A(t_h) )\Big] x(t_h),& \;\;\;   \;\;\;  x(t_0)=x_0 \\
 \rho(t_h) = \Big[\mathbb{I} + \frac{\Delta t}{2}A(t_{h+1})^T+  \frac{\Delta t}{2}  A(t_h)^T (\mathbb{I} + \Delta t \cdot A(t_{h+1})^T)\Big]\rho(& t_{h+1})\\+ \Delta t (\mathbb{I} +  \frac{\Delta t}{2}A(t_h)^T)  Q x(t_{h+1}),  
 \;\;\;\;\;\;\;\;\; \rho(t_{\mathcal{N}})&=P_1 x(t_{\mathcal{N}}) 
\end{split}
\end{equation}
\normalsize
where $\mathbb{I}$ is the $n\times n$ identity matrix and $t_{h+1} = t_h + \Delta t $ with $\Delta t$ the step size. It is assumed for notational simplicity that $A(\cdotp) := \overline{A}(\cdotp,\varSigma,\mathcal{T})$ defined in \eqref{eq16}. Notice that both approximation methods  for the state and co-state, \eqref{discrexpl} and \eqref{discrimpr}, depend on the step size $\Delta t$ as opposed to SIOMS state and co-state expressions \eqref{eq2} and \eqref{rho} that are independent of a step size.  

In the following example, the execution time and error of approximation are measured against the selected number of samples $\mathcal{N}$. The step size $\Delta t$ is constant so that the samples are evenly-spaced. For the Forward and Improved Euler methods, the number of samples $\mathcal{N}$ determines  the fixed step size $\Delta t$ used for online integration of \eqref{eq15} and \eqref{rhoeq} resulting in the approximations \eqref{discrexpl} and \eqref{discrimpr}. However, in SIOMS the number of samples $\mathcal{N}$ does not determine the step size used in the off-line numerical integration---instead, the STM and ATM
equations, \eqref{eq55} and \eqref{eq8}, are numerically solved~\footnote{For this example, equations  \eqref{eq55} and \eqref{eq8} are solved by a fixed-step Improved Euler's method (i.e.\ two-stage Runge Kutta).}and the resulting trajectories are sub-sampled with the desired sampling frequency $1/\Delta t$ to create the final stored data points. Note that we are only able to perform this additional sub-sampling interpolation because it does not affect the total execution time of the online algorithm portion.
The fact that the sub-sampling process is applied on smooth trajectories produced by the continuous vector fields in  \eqref{eq55} and \eqref{eq8}---along with the fact that the expressions for evaluating the state and co-state, \eqref{eq2} and \eqref{rho}, do not depend on any discretization step size---guarantees that approximation accuracy of each sample does not drop as the number of samples decreases.\footnote{Choosing to use interpolating methods might be concerning with regard to approximation accuracy of the full state and co-state trajectories. Regardless, there are two ways to keep approximation errors low: i) by using higher-order interpolating methods and ii) by using a larger number of samples. With SIOMS, we can select a large number of samples without dramatically increasing the execution time (Fig.\ref{dataplot}).}
Regardless of the particular choice of $\Delta t$, the role of $\Delta t$ has the same impact on all three representations  of state and co-state evolution (SIOMS, Forward and Improved Euler)---in each case, $\Delta t$ determines the number of samples $\mathcal{N}$ (that can be) available (without interpolation) during each iteration.
All methods were implemented in MATLAB, on a laptop with an Intel Core i7 chipset.

The results are summarized in Fig. \ref{dataplot}. Figure \ref{dataplot}a illustrates the variation of execution time with respect to the selected number of samples. Execution time refers to the number of seconds required for 10 algorithm iterations---no significant change in cost is observed in subsequent iterations as seen in Fig. \ref{example1}b. In all cases, the final optimal cost was found to be in the range 0.45-0.5. Both Euler methods exhibit a similar rising trend with the execution time reaching a maximum of 13 seconds when 20,001 samples are used (i.e.\ step size of 0.0001 secs). With SIOMS, however, a significantly lower increase rate is observed with a maximum execution time at only approximately 1.3 seconds. The reasoning for this observed difference is that with Euler methods, all samples of the state and co-state trajectories  must be calculated in every iteration whereas in SIOMS one only needs to calculate the state and co-state values necessary for the procedures of the algorithm (e.g. computation of new switching times) using the expressions \eqref{eq2} and \eqref{psi} respectively.

The variation of approximation error with respect to the number of samples is depicted in Fig.  \ref{dataplot}b. Here, by approximation error we refer to the 2-norm of the root-mean-squared (RMS) differences between the analytic and computed state values for all states at sample points. As explained earlier, the error with SIOMS remains approximately zero ($\approx 0.0002$) regardless of the sample size. The trade-off between computation time and approximation error is particularly obvious with the Forward Euler's method, where the error only approaches zero when a maximum number of samples is employed by which time the corresponding execution time is prohibitive. Interestingly, Improved Euler's method starts with a lower error ($\approx 0.07$) and drops to its minimum value of $\approx 0.0002$ when 1600 samples and above are used. With the lowest approximation error ($\approx 0.0002$),  Improved Euler can achieve a minimum execution time of approximately $3$ seconds  compared to  $0.2$ seconds achieved by SIOMS. 
With low execution time ($\approx0.2$ with 100 samples used), Improved Euler can achieve a minimum error close to $0.1$ compared to $0.0002$ achieved by SIOMS.

\subsection{Computational Complexity}

In Section \ref{algor}, we showed that all the state and co-state information needed in Algorithms \ref{sioms} and \ref{rec}, is encoded in the STM, $\Phi^j(t,T_0)$, and ATM, $\Psi^j(t,T_M)$, $\forall j \in \{1,...,N\} $ which are solved for all $t \in [T_0,T_M]$ prior to the optimization routine. Therefore, the calculation of $x^k(t)$ and $P^k(t)$ and consequently the optimality condition $\theta^k$ relies simply on memory calls and matrix algebra. No additional differential equations need to be solved for during optimization. 

The algorithm complexity can be discussed in terms of the number of matrix multiplications involved in each iteration.  Recall that at each iteration, $x(t)$ is given by (\ref{eq2})  and $P(t)$ by (\ref{rho}), but the total number of state and co-state evaluations depends on the number of time instances the descent direction \eqref{modedef} must be evaluated (e.g. for the calculation of $\theta^k$ in \eqref{opt1}). Taking this into consideration, we will look at the algebraic calculations required for the evaluation of the state, co-state and descent direction at a single time instance $t$.

First, for executional efficiency, one may calculate all the state and co-state values at the switching times, $x(T_i)$ and $P(T_i)$, given the current mode schedule $(\varSigma(u^k),\mathcal{T}(u^k))$  at the beginning of each iteration. To compute the state, begin with $x(T_0)=x_0$ and then recursively calculate
\begin{equation}
\label{eq5}
x(T_i)= \Phi^{\sigma_i}(T_i,T_{i-1})x(T_{i-1}) \forall i\in\{1,..., M-1\}.
\end{equation}
Using STM property 2 and following a similar approach as in the derivation of (\ref{eq2}),  this computation comes down to $2(M-1)$ matrix multiplications, assuming that all $\Phi^j(t,T_0)^{-1}$ for all  $j \in \{1,...,N\} $ have also been stored in memory. 
Similarly, begin with $P(T_M)=P_1$ and then recursively calculate  
\begin{equation}
\begin{split}
P(T_{i})= &\Psi^{\sigma_{i+1}}(T_i,T_M)+ \Phi^{\sigma_{i+1}}(T_{i+1},T_i)^T \\&[P(T_{i+1})-\Psi^{\sigma_{i+1}}(T_{i+1},T_M)] \Phi^{\sigma_{i+1}}(T_{i+1},T_{i}) 
\end{split}   
\end{equation}
for all $i\in\{1,..., M-1\}$. Note that the derivation of the above expression is identical to the derivation of (\ref{rho}). Knowing that all $\Phi^{\sigma_{i+1}}(T_{i+1},T_{i})$ have already been calculated in (\ref{eq5}), another $2(M-1)$ multiplications are required for the calculation of $P(t)$.     
To summarize, the standard computational cost of the algorithm comes down to a total of $4(M-1)$ multiplications per iteration. 

Now,  to evaluate equation (\ref{eq2}) and (\ref{rho})  at any random time $t$ during the optimization process, we only need $6$ additional multiplications, $2$ for the state $x(t)$ and $4$ for the relation $P(t)$. Therefore, to evaluate the descent direction at any random time, $9$ multiplications are required in total, including the algebra involved in \eqref{modedef}.

Finally, each iteration of Algorithm \ref{sioms} involves $4(M-1)$ multiplications for the calculation of $x(T_i)$ and $P(T_i)$, and $9 \lambda$ additional multiplications where $\lambda$ is the number of evaluations of the expression  \eqref{modedef} for the descent direction.

\section{Closed-Loop Simulation and Experimental Implementation} 
\label{closed}

\begin{figure}[t]
\centering
\includegraphics[width=3.3in]{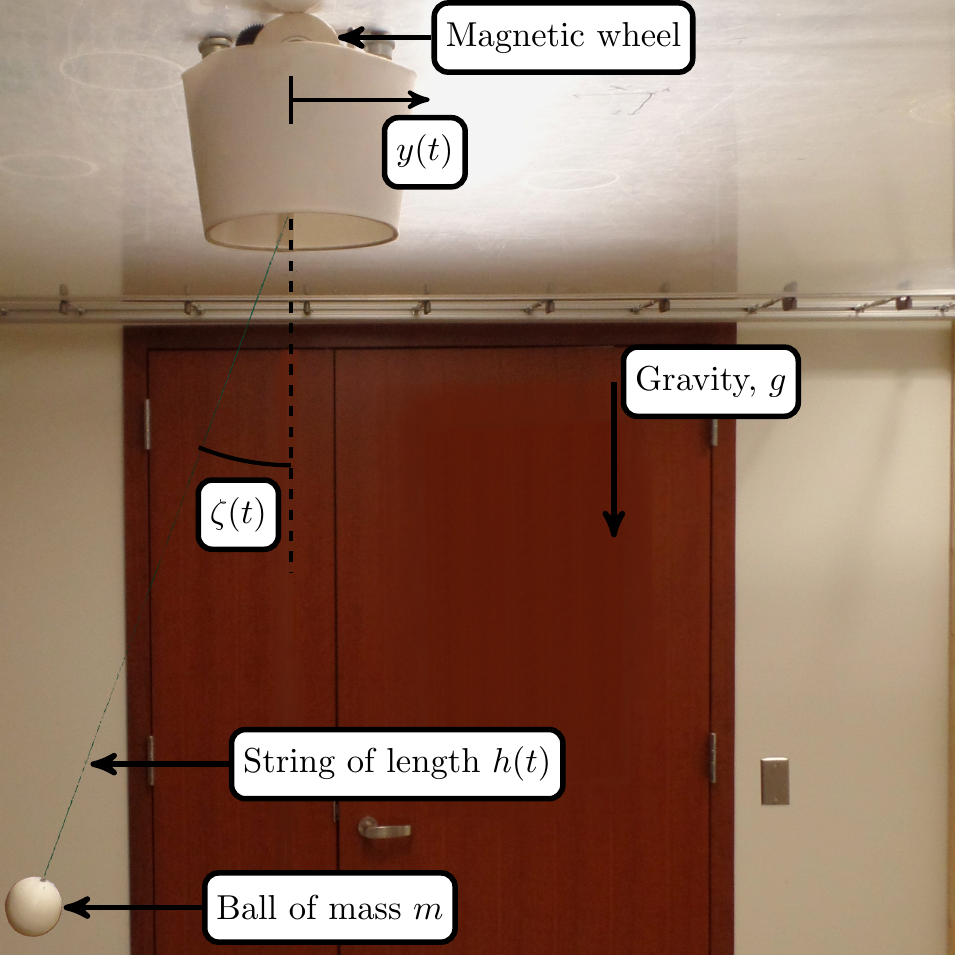} 
      \caption{ The experimental setup consists of an one-dimensional
      differential drive mobile robot with magnetic wheels (i.e.\ cart) and a ball suspended by a string. The string changes length by means of an actuated reeling system attached on the robot. The system configuration is measured by a Microsoft Kinect at $\approx 30Hz$. The full state is estimated using an Extended Kalman Filter. The Robot Operating System (ROS) is used for collecting sensed data and transmitting control signals (i.e.\ robot acceleration values). See more in \cite{943,956}.  
      }
      \label{cart}
  \end{figure}

In this section, SIOMS is implemented in a closed-loop manner (see Algorithm 2) and  is tested on a cart and suspended mass system in simulation and on a customized experimental setup. 

The system model under concern is linear time-varying with two configuration variables, $q(t) = [y(t), \zeta(t)] ^T$,
where $y(t)$ is the horizontal displacement of the cart and $\zeta(t)$ is the
rotational angle of the link as seen in Fig.~\ref{cart}. The length of the link varies with time. Denoting by $h(t)$ the~time-varying string length, $g$ the gravity acceleration and by $m$ and $c$ the mass and damping coefficient, the linearized system equations take the form in \eqref{eq17} with 
\begin{equation}
\label{susp}
A_i(t)=  \begin{pmatrix}
& 0  & 1 &  0  & 0 & 0 &\\
& 0  & 0 &  0  & 0 & -\alpha_i &\\
& 0  & 0 &  0  & 1 & 0 &\\
& 0  & 0 &  -\frac{g}{h(t)}  & -\frac{c}{m h(t)^2} & -\frac{1}{h(t)}\alpha_i &\\
& 0  & 0 &  0  & 0 & 0 &\\
\end{pmatrix}\;\;\;\;,
\end{equation}
\begin{equation}
\label{length}
h(t) = sin(t)+2
\end{equation}
and $N=3$ i.e.\ three possible modes. The cart's horizontal acceleration $\alpha$ is directly controlled and can switch between
the values $\alpha_1=0$, $\alpha_2=-0.5$ and  $\alpha_3=0.5$. Notice we have augmented the state-space from $\mathbb{R}^4$ to $\mathbb{R}^5$ in order to transform the originally affine model to the linear form in \eqref{eq17}. The augmented state vector is $x=[y,\dot{y},\zeta,\dot{\zeta},\tilde{u}]$ where $\tilde{u}$ is the auxiliary state variable. 
System parameters are defined as $m=0.124$, $c=0.05$ and $g=9.8$.

Our objective is to find the mode schedule that minimizes the angle oscillation while the cart remains in a neighborhood near the origin and is accordingly characterized by  the quadratic cost functional \eqref{cost} with $Q=diag[0,0,10,1,0]$ and $P_1=diag[0.1,0.01,10,1,0]$. The system starts at an initial condition $x_0=[0.5,0,0.1,0,1]^T$.

\begin{figure}[t]
\centering
\includegraphics[width=3.3in]{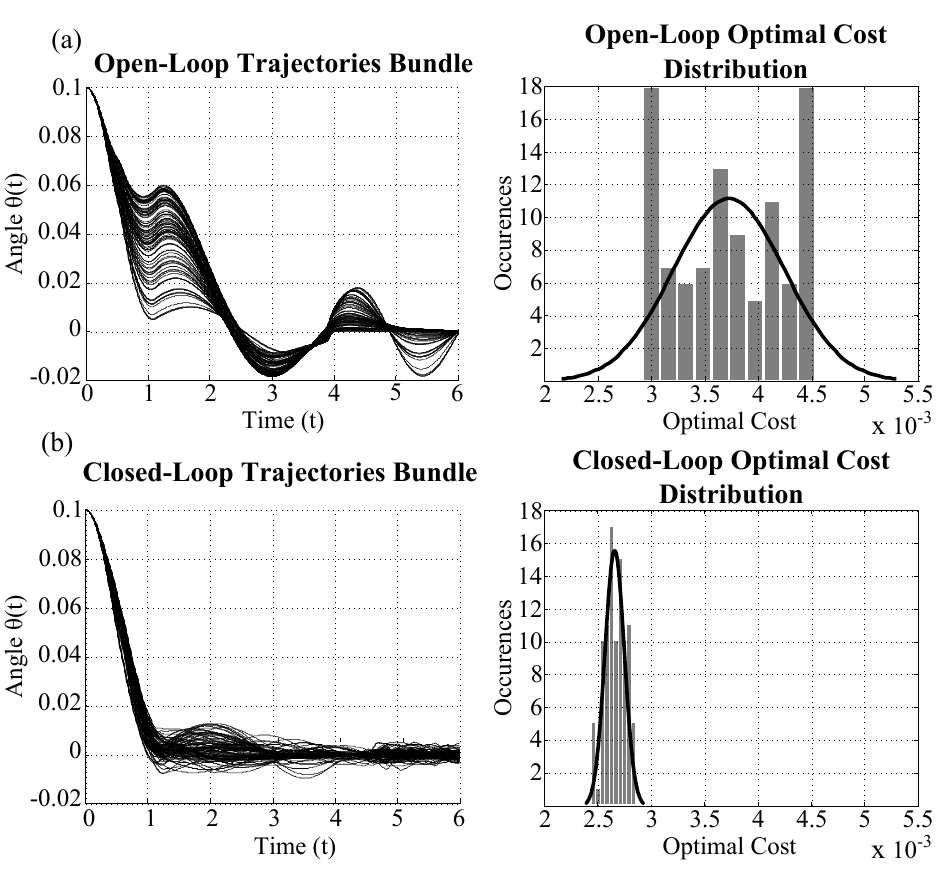} 
      \caption{ Robustness to uncertainty in the damping coefficient through Monte-Carlo analysis. Angle trajectories bundle and optimal cost distribution for (a) open-loop SIOMS (Algorithm \ref{sioms} with $T_0 = 0$ and $T_M = 6$) and (b) closed-loop SIOMS (Algorithm \ref{rec} with $\delta=0.2$ and $T=3$). 
      }
      \label{monte}
  \end{figure}

\subsection{ Simulation Results } 
We apply Algorithm \ref{rec} to the optimal control problem stated previously and compare its performance with Algorithm \ref{sioms} in terms of (1) disturbance rejection and (2) robustness to system parameter uncertainties. For real-time SIOMS execution, both algorithms were now implemented in Python.
\begin{figure*}[t]
\centering
\includegraphics[width=6.53in]{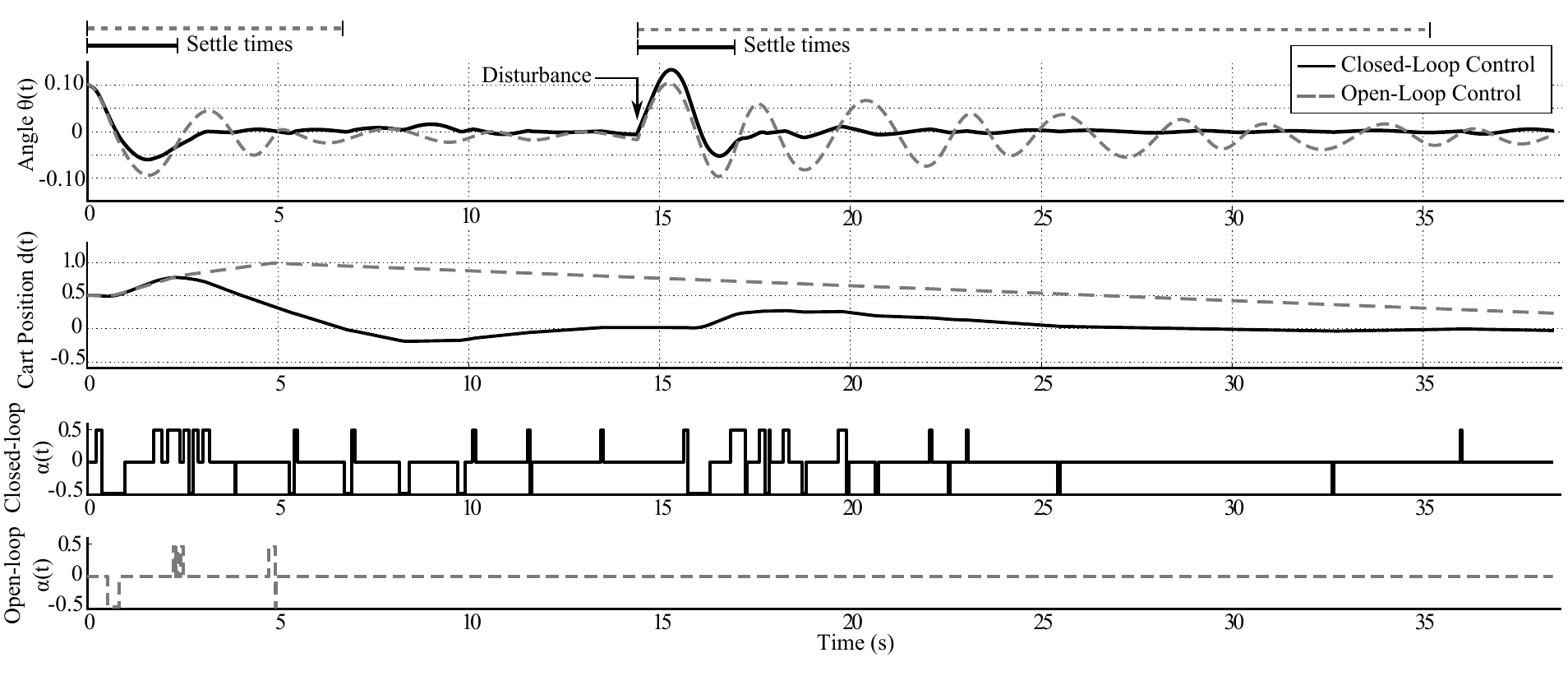} 
      \caption{ Open-loop SIOMS (Algorithm \ref{sioms} with $T_M=40s$) vs Closed-loop SIOMS (Algorithm \ref{rec} with $\delta=0.5$ and $T=3s$) in simulation.}
      \label{recfig}
\end{figure*} 
\paragraph{Disturbance rejection}
We ran Algorithm \ref{rec}  with parameters $\delta=0.5$ and $T=3s$ for a total of 40 seconds.  A disturbance is applied at time $\approx 14s$. Each run of Algorithm \ref{sioms} (i.e.\ 5 SIOMS iterations) lasted on average $0.04s$ of CPU time. Note that the algorithm was implemented in a real-time manner---starting the system simulation/integration from $t=0s$, a new switching control is calculated and applied every $\delta=0.5$ seconds using information about the current system state. For comparison, we additionally ran a one-time open-loop SIOMS (Algorithm \ref{sioms}) with $T_0 =0s$ and $T_M=40s$. The cost is reduced from $J_0\approx 1.96$ to $J\approx 0.58$ after 15 iterations; the optimal switching control was pre-calculated and later applied to the system. 

The results are illustrated in Fig. \ref{recfig}. Starting at an initial value of $0.1rad$, the angle has a settle time\footnote{Settle time is defined here as the time from the arrival of the disturbance until the angle reaches and stays within the settle boundary from $-0.025 rad$ to $0.025 rad$ surrounding the origin.} of about $2.5s$ with closed-loop control compared to $5s$ when open-loop SIOMS is applied. As expected, the disturbance triggers a high angle oscillation with a settle time $>20s$, as the effect is not taken into account by the open-loop controller. The receding-horizon SIOMS, however, results in a much lower settle time of $2.5s$, providing an efficient real-time response to the random disturbance.  
The last 2 diagrams in Fig. \ref{recfig} show the switched cart acceleration $\alpha$ with respect to time as calculated by each algorithm. In close-loop control where the most reliable performance is observed, a total of $65$ switches occur with an average mode duration of $\approx0.42s$ and a minimum mode duration (i.e.\ period during which the mode remain fixed) of $\approx0.02s$ representing actuator frequency limits.

\paragraph{Robustness to model uncertainties}

In a subsequent comparison, we examine the robustness of Algorithm \ref{rec} to model uncertainties and compare its performance to Algorithm \ref{sioms}. In particular, we perform a Monte-Carlo analysis where both algorithms are run 100 times in the following scheme: the optimal switching control is calculated using the system model in \eqref{susp} and is subsequently applied to an equivalent system with randomly added noise in the damping parameter, i.e. $c_{actual}= 0.05 + \omega$ where $\omega$ is a random real number in the range $[-0.05,3.0]$ so that $c_{actual}\in[0,3.05]$. 

We demonstrate the results in Fig. \ref{monte}. The diagrams on the left show the resulting angle trajectories for $t\in [0,6]$ of all algorithm runs. It can be observed that open-loop SIOMS is more sensitive to changes in the damping coefficient  compared to closed-loop SIOMS that exhibits a more robust performance. 
The distribution of optimal costs across all runs  is given in the remaining diagrams of Fig. \ref{monte}. For both open-loop and closed-loop SIOMS, the optimal cost is calculated as in \eqref{cost} over the resulting trajectories $x(t)$ for all $t \in [0,6]$.
The mean optimal cost in open-loop SIOMS is $\approx~0.0037$ compared to $\approx~0.0027$ in the closed-loop implementation. In addition, with receding-horizon SIOMS (Algorithm \ref{rec}) the standard deviation is $0.08 \cdot 10^{-3}$ which is significantly lower than the standard deviation $0.51 \cdot 10^{-3}$ observed in open-loop SIOMS (Algorithm~\ref{sioms}).

\subsection{Experimental Results }

\begin{figure}[t]
\centering
\includegraphics[width=3.3in]{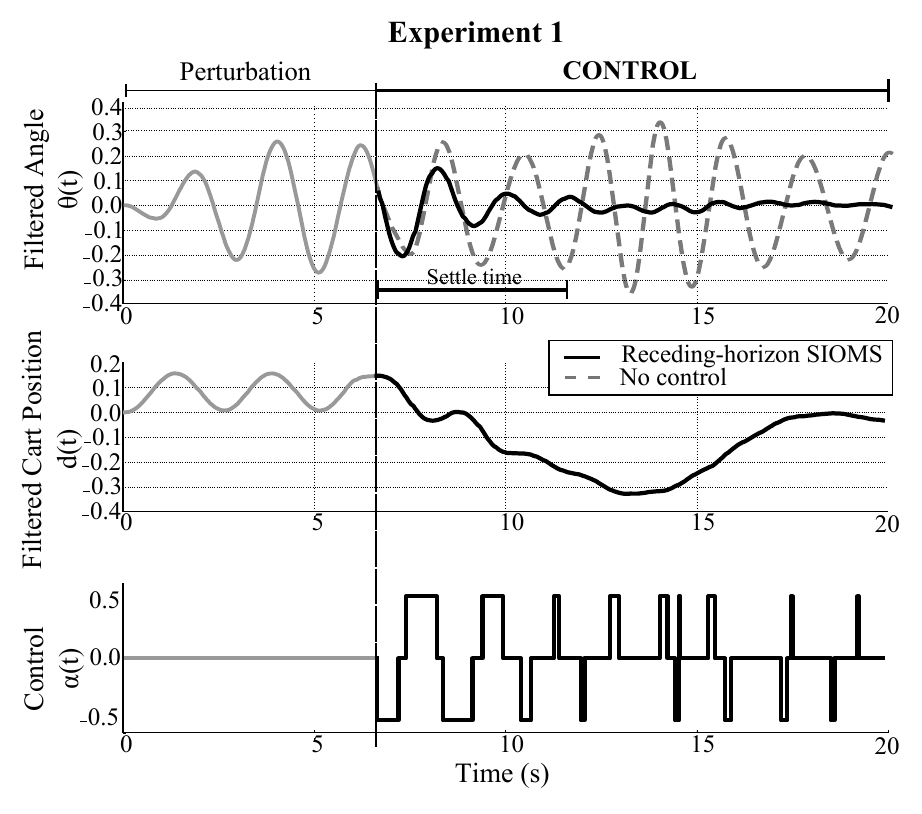} 
      \caption{ An example trial of Experiment 1. First, the robot follows a sinusoidal trajectory perturbing the string angle. Approximately $6.6$ seconds later, receding-horizon SIOMS is applied in real time and drives the angle back to the origin in approximately $4.8$ seconds. Without control, the angle exhibits high oscillations with minimal decay.}
      \label{exp1}
\end{figure}
 
\begin{figure*}[t]
\centering
\includegraphics[width=6.6in]{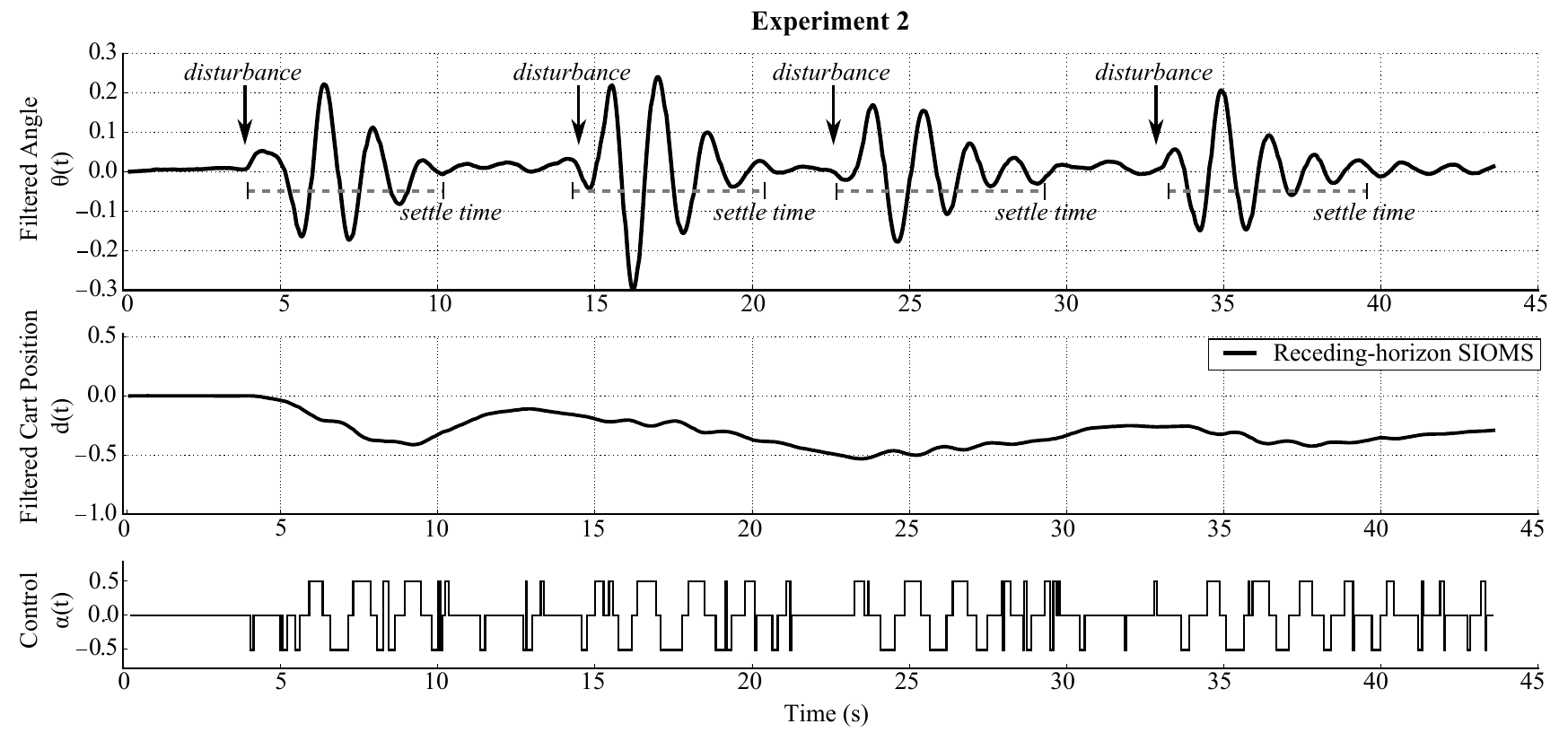} 
      \caption{ An example trial of Experiment 2. SIOMS controller is always active while a person applies random disturbances by pushing the suspended ball four times sequentially. The controller reacts in real time to regulate the angle. Approximate settle time is at $6$ seconds.}
      \label{exp2}
\end{figure*} 

In this paragraph, the performance of the closed-loop hybrid controller (Algorithm~\ref{rec}) is evaluated experimentally on a real cart and suspended mass system (Fig.~\ref{cart}). More information about this experimental platform can be found in \cite{943,956}. Due to geometric constraints and model discrepancies, a few changes in the parameters were made as follows: $h(t)=0.4 sin(t)+1 $ in \eqref{length}, $c=0.001$ in \eqref{susp}, $\delta=0.4$ and $T=5s$ in Algorithm \ref{rec}.   The same objective as in simulation was pursued   i.e.\ real-time angle regulation with the robot position maintained close to the origin. The weight matrices in  \eqref{cost} were set as $Q=diag[0,0,1000,0,0]$ and $P_1=diag[1,0,100,0,0]$. 
We ran 2 sets of experiments to illustrate the features of the hybrid controller based on Algorithm \ref{rec}.  

In Experiment 1, the SIOMS controller is initially inactive and we perturb the string angle by setting a predefined oscillatory trajectory to the cart/robot. After approximately $6.6$ seconds, the controller is activated to optimally drive the angle to zero using Algorithm \ref{rec}. A video of the experiment is available in \cite{sioms1}. One example trial of Experiment 1 is illustrated in Fig. \ref{exp1}. During the perturbation, the angle exhibits an oscillatory response with peak amplitude at $0.25rad$.  Once receding-horizon SIOMS is applied, the string angle starts approaching the origin with a settle time of $4.8$ seconds and the robot moves slightly to the left before returning to the origin. 
For comparison purposes,  Fig. \ref{exp1} also shows the angle trajectory for the case when no control was applied following the perturbation (i.e. $\alpha(t)=0$). One may observe that the uncontrolled system is highly underdamped with no settle time achieved in a time horizon of $\approx 14$ seconds.  Note that the sinusoidal change in peak amplitude and frequency is a result of the time-varying string length.  

\begin{table*}[t]
\caption{We ran 12 trials of Experiment 1 with 4 different perturbation levels. }
\label{exp_table}
\centering
\resizebox{6.0in}{!}{
\begin{tabular}{c|c|c|c||c|c|c||c|c|c||c|c|c|}
\cline{2-13}
&  \multicolumn{3}{ c|| }{\textbf{Perturbation 1}} &
\multicolumn{3}{ c|| }{\textbf{Perturbation 2}} &
\multicolumn{3}{ c|| }{\textbf{Perturbation 3}} &
\multicolumn{3}{ c| }{\textbf{Perturbation 4}}  \\ 
&  \multicolumn{3}{ c|| }{peak angle = $0.18rad$} &
\multicolumn{3}{ c|| }{peak angle = $0.25rad$} &
\multicolumn{3}{ c|| }{peak angle = $0.33rad$} &
\multicolumn{3}{ c| }{peak angle = $0.4rad$}  \\ \cline{2-13}
\multicolumn{1}{ r| } {\textit{Trial}} & \textit{1} & \textit{2} & \textit{3} & \textit{1} & \textit{2} & \textit{3} & \textit{1} & \textit{2} & \textit{3} & \textit{1} & \textit{2} & \textit{3} \\ \cline{1-13}
\multicolumn{1}{ |c| } {\textbf{switches / second}} & 2.80 & 3.20 & 2.60 & 2.36 & 2.61 & 2.80 & 3.88 & 2.77 & 2.46 & 2.89 & 2.59 & 2.66 \\ \cline{1-13}
\multicolumn{1}{ |c| } {\textbf{average mode duration (s)}} & 0.34 & 0.27 & 0.32 & 0.31 & 0.35 & 0.32 & 0.38 & 0.33 & 0.35 & 0.31 & 0.32 & 0.34 \\ \cline{1-13}
\multicolumn{1}{ |c| } {\textbf{settle time (s)}} & 2.9 & 3.6 & 4.2 & 3.4 & 4.2 & 4.8 & 7.2 & 7.5 & 6.9 & 9.4 & 8.6 & 9.1 \\ \cline{1-13}
\end{tabular}
}
\end{table*}
 
We repeated Experiment 1 for four different perturbation levels, each characterized by the peak angle amplitude achieved. Three trials per perturbation were run i.e.\ twelve trials in total. As performance metrics, we used a) the number of switches per second, b) the average mode duration and c) the settle time for the string angle.  Our goal was to verify the reliability and efficacy of the controller in noisy conditions induced by sensor and model deficiencies. The results are given in Table \ref{exp_table}. Throughout the trials, the number of switches per second ranges from $2.3$ to $3.8$. The average mode duration also exhibits low variation among different trials with a range from $0.27$ to $0.38$ seconds. As expected, settle times increase with higher perturbation levels but remain fairly close among trials of the same perturbation. 

In Experiment 2, we sought to evaluate the performance of the hybrid controller when random disturbances occur in real time. To achieve this, the experiment is initialized at $y=0$, $\zeta=0$, $\alpha=0$ and zero velocities. With the receding-horizon SIOMS controller activated, a person pushes the ball (i.e.\ suspended mass) towards one direction to create real-time disturbances. The controller responds to the disturbance to regulate the angle and drive it back to zero. A video of the experiment is available in \cite{sioms2}. An example trial of Experiment 2 is presented in Fig. \ref{exp2} where four consecutive disturbances of varied amplitudes are applied.  One may observe that the controller regulates the angle  with settle times of approximately $6$ seconds in all four cases. Furthermore, as a result of the terminal cost applied at $y(t)$, the robot does not deviate significantly from the origin and always returns close to zero once the angle is near zero.

\section{ Conclusions and Future Work }

Our objective in this paper is to achieve fast and consistent, real-time hybrid control by taking advantage of linearity of a switched system. In general, mode scheduling is challenging due to the fact that both the mode sequence and the set of switching times must be optimized jointly. Thus, execution time of an optimization is often prohibitive for real-time applications and can only be reduced at the expense of approximation accuracy. In addition, the numerical implementation of optimal mode scheduling algorithms requires consistent solution approximations that are prone to numerical errors due to discontinuities of the switched system under concern.  

We addressed these issues by introducing an algorithm (SIOMS) for scheduling the modes of linear time-varying switched systems subject to a quadratic cost functional. By solving a single set of differential equations off-line, open-loop SIOMS requires no online simulations while closed-loop SIOMS only involves an integration over a limited time interval rather than the full time horizon. The proposed algorithm is fast and free of the trade-off between execution time and approximation errors. Furthermore, in practical implementation, the proposed solution of the state and adjoint equations is independent of the selected step size. For this reason, approximation accuracy and consistency in SIOMS does not depend on the number of samples used for approximation of the state and co-state trajectories. We verified the aforementioned advantages using a numerical example and comparing SIOMS to algorithms that include common integration schemes. Finally, to verify the efficacy of receding-horizon SIOMS in real-world applications, we performed a real-time experiment using ROS. Our experimental work demonstrated that a cart and suspended mass system can be regulated in real time using closed-loop hybrid control signals.

Future work will focus on formally establishing stability of the receding-horizon SIOMS controller. Stability criteria for model predictive hybrid control may rely on hysteresis and dwell-time conditions (see \cite{mayne2014model, magni2001output, muller2012improving,sanfelice2014input,HSCC10}).

\section*{References}

\bibliographystyle{IEEEtran}
\bibliography{references}

\end{document}